\documentclass{amsart}
\usepackage{amssymb,amsmath,latexsym}

\newtheorem{theorem}{Theorem}[section]

\newtheorem{proposition}[theorem]{Proposition}

\theoremstyle{definition}

\theoremstyle{remark}

\numberwithin{equation}{section}

\newcommand{\bea}{\begin{eqnarray}}
\newcommand{\eea}{\end{eqnarray}}
\newcommand{\ba}{\begin{array}}
\newcommand{\ea}{\end{array}}

\input{tcilatex}

\begin{document}
\title[On Some Geometric And Topological Properties of Sequence Spaces]{Some
Geometric And Topological Properties Of\\
A New Sequence Space Defined By\\
de la Vall\'{e}e-Poussin Mean}
\author{NEC\.{I}P \c{S}\.{I}M\c{S}EK$^*$}
\address{ ADIYAMAN UNIVERSITY\\
Faculty of Arts and Sciences, Department of Mathematics-ADIYAMAN-TURKEY}
\email{necsimsek@yahoo.com}
\author{Ekrem SAVA\c{S}}
\address{\.{I}STANBUL COMMERCE UNIVERSITY, Department of Mathematics,
Uskudar, \.{I}stanbul, TURKEY}
\email{ekremsavas@yahoo.com}
\author{Vatan Karakaya}
\address{YILDIZ TECHNICAL UNIVERSITY\\
Department of Mathematical Engineering\\
Davutpasa Campus, Esenler\\
\.{I}STANBUL-TURKEY}
\email{vkkaya@yahoo.com}
\keywords{de la Vall\'{e}e-Poussin, Ces\'{a}ro sequence spaces, H-property,
Banach-Saks property, geometrical properties.\\
\ $^*$Corresponding Author: necsimsek@yahoo.com}

\begin{abstract}
The main purpose of this paper is to introduce a new sequence space by using
de la Vall\'{e}e-Poussin \ mean and investigate both the modular structure
with some geometric properties and some topological properties with respect
to the Luxemburg norm.
\end{abstract}

\maketitle

\section{Introduction}

In summability theory, de la Vall\'{e}e-Poussin's mean is first used to
define the $(V,\lambda )$-summability by Leindler \cite{Leindler-1965}.
Malkowsky and Sava\c{s} \cite{MalSav-2000} introduced and studied some
sequence spaces which arise from the notion of generalized de la Vall\'{e}%
e-Poussin mean. Also the $(V,\lambda )$-summable sequence spaces have been
studied by many authors including \cite{Mik2006} and \cite{Savsav-2003}.

Recently, there has been a lot of interest in investigating geometric
properties of several sequence spaces. Some of the recent work on sequence
spaces and their geometrical properties is given in the sequel: Shue \cite%
{Shue-1970} first defined the Ces\'{a}ro sequence spaces with a norm. In
\cite{Liu-1996}, it is shown that the Ces\'{a}ro sequence spaces $ces_{p}$ $%
\left( 1\leq p<\infty \right) $ have Kadec-Klee and Local Uniform
Rotundity(LUR) properties. Cui-Hudzik-Pluciennik \cite{Cui-Hud-Plu-1997}
showed that Banach-Saks of type $p$ property holds in these spaces. In \cite%
{MurFeyz-2006}, Mursaleen et al studied some geometric properties of normed
Euler sequence space. Karakaya \cite{Karakaya2007} defined a new sequence
space involving lacunary sequence space equipped with the Luxemburg norm and
studied Kadec-Klee($H$), rotund($R$) properties of this space. Quite
recently, Sanhan and Suantai \cite{San-Suan-2003} generalized normed Ces\'{a}%
ro sequence spaces to paranormed sequence spaces by making use of K\"{o}the
sequence spaces. They also defined and investigated modular structure and
some geometrical properties of these generalized sequence spaces. In
addition, some related papers on this topic can be found in \cite{Basar},%
\cite{Chen},\cite{Diestel},\cite{Mu-2007},\cite{Mus-1983} and \cite%
{Nec-Vat-2009}.

In this paper, our purpose is to introduce a new sequence space defined by
de la Vall\'{e}e-Poussin's mean \ and investigate some topological and
geometric properties of this space.

The organization of our paper is as follows:\ In the first section, we
introduce some definition and concepts that are used throughout the paper.
In the second section, we construct a new paranormed sequence space and
investigate some geometrical properties of this space. Finally, in the third
section, we construct the modular space $V_{\rho }(\lambda ;p)$ which is
obtained by paranormed space $V(\lambda ;p)$\ and we investigate the
Kadec-Klee property of this space. We also show that the modular space $%
V_{\rho }(\lambda ;p)$\ is a Banach space under the Luxemburg norm. Also in
this section, we investigate the Banach-Saks of type $p$ property of the
space $V_{p}(\lambda ).$

\section{Preliminaries, Background and Notation}

The space of all real sequences $x=(x(i))_{i=1}^{\infty }$ is denoted by $%
\ell ^{0}$. Let $\left( X,\left\Vert .\right\Vert \right) $ (for the brevity
$X=\left( X,\left\Vert .\right\Vert \right) $ ) be a normed linear space and
let $S(X)$ and $B(X)$ be the unit sphere and unit ball of $X$, respectively.

A Banach space $X$ which is a subspace of $\ell ^{0}$ is said to be a K\"{o}%
the sequence space, if $\left( \text{see \cite{Linden-1977}}\right) ;$

$\left( i\right) $ for any $x\in \ell ^{0}$ and $y\in X$ such that $%
\left\vert x(i)\right\vert \leq \left\vert y(i)\right\vert $ for all $i\in
\mathbb{N}$, we have $x\in X$ \ and $\left\Vert x\right\Vert \leq \left\Vert
y\right\Vert ,$

$\left( ii\right) $ there is $x\in X$ with $x(i)>0$ for all $i\in \mathbb{N}%
. $

We say that $x\in X$ is order continuous if for any sequence \ $\left(
x_{n}\right) $ in $X$ such that $x_{n}(i)\leq \left\vert x(i)\right\vert $
for each $i\in \mathbb{N}$ and $x_{n}(i)\rightarrow 0\left( n\rightarrow
\infty \right) $, $\left\Vert x_{n}\right\Vert \rightarrow 0$\ holds$.$ A K%
\"{o}the sequence space $X$ is said to be order continuous if all sequences
in $X$ are order continuous. It is easy to see that $x\in X$ is order
continuous if and only if $\left\Vert
(0,0,...,0,x(n+1),x(n+2),...)\right\Vert \rightarrow 0$ as $n\rightarrow
\infty $.

A Banach space $X$ is said to have the \textit{Kadec-Klee property} (or
property ($H$)) if every weakly convergent sequence on the unit sphere with
the weak limit in the sphere is convergent in norm.

Let $1<p<\infty .$ A Banach space is said to have the $Banach-Saks$ $type$ $%
p $ or property $(BS_{p})$, if every weakly null sequence $(x_{k})$ has a
subsequence $(x_{k_{l}})$ such that for some $C>0,$%
\begin{equation*}
\left\Vert \sum\limits_{l=0}^{n}x_{k_{l}}\right\Vert <C(n+1)^{\frac{1}{p}}
\end{equation*}%
for all $n\in \mathbb{N}$ (see \cite{Knau-1992}).

For a real vector space $X$, a functional $\rho :X\rightarrow \lbrack
0,\infty ]$ is called a \textit{modular} if it satisfies the following
conditions:

\qquad i) $\rho (x)=0\Leftrightarrow x=0,$

\qquad ii) $\rho $($\alpha x)=\rho (x)$ for all $\alpha \in \mathbb{F}$ \
with $|\alpha |=1,$

\qquad iii) $\rho (\alpha x+\beta y)\leq \rho (x)+\rho (y)$ for all $x,y\in
X $ and all $\alpha ,\beta \geq 0$ with $\alpha +\beta =1.$

Further, the modular $\rho $ is called \textit{convex} if

\qquad iv) $\rho (\alpha x+\beta y)\leq \alpha \rho (x)+\beta \rho (y)$
holds for all $x,y\in X$ and all $\alpha ,\beta \geq 0$ with $\alpha +\beta
=1.$

$\rho $ is a modular in $X$, we define%
\begin{equation*}
X_{\rho }=\left\{ x\in X:\text{ }\rho (\lambda x)\rightarrow 0\text{ \ \ as
\ \ \ }\lambda \rightarrow 0^{+}\right\},
\end{equation*}%
\begin{equation*}
X_{\rho }^{\ast }=\left\{ x\in X:\text{ }\rho (\lambda x)<\infty \text{\ for
some\ \ \ }\lambda >0\right\} .
\end{equation*}%
It is clear that $X_{\rho }\subseteq X_{\rho }^{\ast }$. If $\rho $\ is a
convex modular, for $x\in X_{\rho }$, we define
\begin{equation*}
||x||_{L}=\inf \left\{ \lambda >0:\rho (\frac{x}{\lambda })\leq 1\right\}
\end{equation*}%
and%
\begin{equation*}
||x||_{A}=\inf_{\lambda >0}\frac{1}{\lambda }\left( 1+\rho (\lambda
x)\right) .
\end{equation*}%
If $\rho $\ is a convex modular on $X$, then $X_{\rho }=X_{\rho }^{\ast }$
and both $||\cdot ||_{L}$ and $||\cdot ||_{A}$ is a norm on $X_{\rho }$ for
which $X_{\rho }$ is a Banach space.

The norms $||\cdot ||_{L}$ and $||\cdot ||_{A}$ are called the \textit{%
Luxemburg norm} and the \textit{Amemiya norm(Orlicz norm)}, respectively.

In addition

\begin{equation*}
||x||_{L}\leq ||x||_{A}\leq 2||x||_{L}
\end{equation*}%
for all $x\in X_{\rho }$ holds (see \cite{Orlicz-1990}).

A sequence $(x_{n})$ of elements of $X_{\rho }$ is called \textit{modular
convergent} to $\ x\in X_{\rho }$ $\ \ $if there exists a $\ \lambda >0$
such that $\ \rho (\lambda (x_{n}-x))\rightarrow 0$ \ as $n\rightarrow
\infty .$

\begin{proposition}
\label{prop1.1}Let \ $(x_{n})\subset X_{\rho }$. Then $||x_{n}||_{L}%
\rightarrow 0$ $($or equivalently $||x||_{A}\rightarrow 0)$ if and only if \
$\rho (\lambda (x_{n}))\rightarrow 0$ \ as $n\rightarrow \infty $, for every
$\lambda >0.$
\end{proposition}

\begin{proof}
See \cite[p.15, Th.1]{Orlicz-1990}.
\end{proof}

Throughout the paper, the sequence $p=(p_{k})$ is a bounded sequence of
positive real numbers with $p_{k}>1$, also $H=\sup_{k}p_{k}$ and $M=\max
\{1,H\}.$

Besides, we will need the following inequalities in the sequel;%
\begin{equation}
\left\vert a_{k}+b_{k}\right\vert ^{p_{k}}\leq K\left(
|a_{k}|^{p_{k}}+|b_{k}|^{p_{k}}\right)  \label{1.1}
\end{equation}%
\begin{equation}
\left\vert a_{k}+b_{k}\right\vert ^{t_{k}}\leq
|a_{k}|^{t_{k}}+|b_{k}|^{t_{k}}  \label{1.2}
\end{equation}%
where $t_{k}=\frac{p_{k}}{M}\leq 1$ and $K=\max \{1,2^{H-1}\}$, with $%
H=\sup_{k}p_{k}.$

Now we begin the construction of a new sequence space.

Let $\Lambda =(\lambda _{k})$ be a nondecreasing sequence of positive real
numbers tending to infinity and let $\lambda _{1}=1$ and $\lambda _{k+1}\leq
\lambda _{k}+1.$

The generalized de la Vall\'{e}e-Poussin means of a sequence $x=\left(
x_{k}\right) $ are defined as follows:%
\begin{equation*}
t_{k}(x)=\frac{1}{\lambda _{k}}\dsum\limits_{j\in I_{k}}x_{k}\text{ \ where }%
I_{k}=[k-\lambda _{k}+1,k]\ \ \ \ \text{for \ \ \ }k=1,2,...\text{ }.
\end{equation*}%
We write%
\begin{equation*}
\lbrack V,\lambda ]_{0}=\left\{ x\in \ell ^{0}:\underset{k\rightarrow \infty
}{\lim }\frac{1}{\lambda _{k}}\dsum\limits_{j\in I_{k}}|x_{j}|=0\right\}
\end{equation*}%
\begin{equation*}
\lbrack V,\lambda ]=\left\{ x\in \ell ^{0}:x-le\in \lbrack V,\lambda ]_{0},%
\text{ for some }l\in \mathbb{C}\right\}
\end{equation*}%
and%
\begin{equation*}
\lbrack V,\lambda ]_{\infty }=\left\{ x\in \ell ^{0}:\underset{k}{\sup }%
\frac{1}{\lambda _{k}}\dsum\limits_{j\in I_{k}}|x_{j}|<\infty \right\}
\end{equation*}%
for the sequence spaces that are strongly summable to zero, strongly
summable and strongly bounded by the de la Vall\'{e}e-Poussin method (see
\cite{Leindler-1965}). In the special case where $\lambda _{k}=k$ for $%
k=1,2,...$ the spaces $[V,\lambda ]_{0},$ $[V,\lambda ]$ and $[V,\lambda
]_{\infty }$ reduce to the spaces $w_{0},$ $w$ and $w_{\infty }$ introduced
by Maddox \cite{Maddox-1}.

We now define the following new paranormed sequence space:%
\begin{equation*}
V(\lambda ;p):=\left\{ x=(x_{j})\in \ell ^{0}:\dsum\limits_{k=1}^{\infty
}\left( \frac{1}{\lambda _{k}}\dsum\limits_{j\in I_{k}}|x_{j}|\right)
^{p_{k}}<\infty \right\} .
\end{equation*}%
The space $V(\lambda ;p)$ is reduced to some special sequence spaces
corresponding to special cases of sequence $(\lambda _{k})$ and $\left(
p_{k}\right) $. For example:\ If we take $\lambda _{k}=k,$ we obtain the
space $ces(p)$ defined by \cite{Suant2003}. If we take $\lambda _{k}=k$ and $%
p_{k}=p$ for all $k\in \mathbb{N}$, the space $V(\lambda ;p)$ reduces to $\ $%
the space $ces_{p}$ defined by \cite{Shue-1970}.

\section{Some Topological Properties Of The Sequence Space $V(\protect%
\lambda ;p)$}

In this section, we will give the topological properties of the space $%
V(\lambda ;p).$ We begin by obtaining the first main result.

\begin{theorem}
\label{theo2.1}a) The space $V(\lambda ;p)$ is a complete paranormed space
with paranorm defined by%
\begin{equation}
h(x):=\left( \dsum\limits_{k=1}^{\infty }\left( \frac{1}{\lambda _{k}}%
\dsum\limits_{j\in I_{k}}|x_{j}|\right) ^{p_{k}}\right) ^{\frac{1}{M}}.
\label{2.1}
\end{equation}%
b) if $p_{k}=p;$ the space $V(\lambda ;p)$ reduced to $V_{p}(\lambda )$
defined by%
\begin{equation*}
V_{p}(\lambda ):=\left\{ x=(x_{j})\in \ell ^{0}:\dsum\limits_{k=1}^{\infty
}\left( \frac{1}{\lambda _{k}}\dsum\limits_{j\in I_{k}}|x_{j}|\right)
^{p}<\infty \right\} .
\end{equation*}%
And the space $V_{p}(\lambda )$ is a complete normed space defined by%
\begin{equation*}
\left\Vert x\right\Vert _{V_{p}(\lambda )}:=\left(
\dsum\limits_{k=1}^{\infty }\left( \frac{1}{\lambda _{k}}\dsum\limits_{j\in
I_{k}}|x_{j}|\right) ^{p}\right) ^{\frac{1}{p}}\text{ \ \ \ \ \ \ }%
(1<p<\infty ).
\end{equation*}

\begin{proof}
a)\ The linearity of $V(\lambda ;p)$ with respect to coordinatewise addition
and scalar multiplication follows from the inequality (\ref{1.1}). Because,
for any $x,y\in V(\lambda ;p)$ the following inequalities are satisfied:%
\begin{equation*}
(2.2)\text{ \ \ }\left( \dsum\limits_{k=1}^{\infty }\left( \frac{1}{\lambda
_{k}}\dsum\limits_{j\in I_{k}}|x_{j}+y_{j}|\right) ^{p_{k}}\right) ^{\frac{1%
}{M}}\leq \left( \dsum\limits_{k=1}^{\infty }\left( \frac{1}{\lambda _{k}}%
\dsum\limits_{j\in I_{k}}|x_{j}|\right) ^{p_{k}}\right) ^{\frac{1}{M}%
}+\left( \dsum\limits_{k=1}^{\infty }\left( \frac{1}{\lambda _{k}}%
\dsum\limits_{j\in I_{k}}|y_{j}|\right) ^{p_{k}}\right) ^{\frac{1}{M}}
\end{equation*}%
and for any $\alpha \in \mathbb{R}$ $($see \cite{Maddox-3}$)$ we have%
\begin{equation}
\left\vert \alpha \right\vert ^{p_{k}}\leq \max \{1,|\alpha |^{M}\}.
\tag{2.3}  \label{2.3}
\end{equation}%
It is clear that $h(\theta )=0$ and $h(x)=h(-x)$ for all $x\in V(\lambda ;p)$%
. Again the inequalities (2.2) and (\ref{2.3}) yield the subadditivity of $h$
and%
\begin{equation*}
h(\alpha x)\leq \max \{1,|\alpha |\}h(x).
\end{equation*}%
Let $(x^{m})$ be any sequence of points of the space $V(\lambda ;p)$ such
that $h(x^{m}-x)\rightarrow 0$ and $(\alpha _{n})$ also be any sequence of
scalars such that $\alpha _{n}\rightarrow \alpha .$ Then, since the
inequality%
\begin{equation*}
h(x^{m})\leq h(x)+h(x^{m}-x)
\end{equation*}%
holds by subadditivity of $h$, the sequence $\left( h(x^{m})\right) _{m\in
\mathbb{N}}$ is bounded and we thus have%
\begin{eqnarray*}
h(\alpha _{m}x^{m}-\alpha x) &=&\left( \dsum\limits_{k=1}^{\infty }\left(
\frac{1}{\lambda _{k}}\dsum\limits_{j\in I_{k}}|\alpha _{m}x_{j}^{m}-\alpha
x_{j}|\right) ^{p_{k}}\right) ^{\frac{1}{M}} \\
&\leq &\left\vert \alpha _{m}-\alpha \right\vert h(x^{m})+|\alpha
|h(x^{m}-x).
\end{eqnarray*}%
The last expression tends to zero as $m\rightarrow \infty ,$ that is, the
scalar multiplication is continuous. Hence $h$ is paranorm on the space $%
V(\lambda ;p)$.

It remains to prove the completeness of the space $V(\lambda ;p)$.

Let $(x^{n})$ be any Cauchy sequence in the space $V(\lambda ;p)$, where $%
x=(x_{j}^{n})=(x_{1}^{n},x_{2}^{n},x_{3}^{n},...)$. Then, for a given $%
\varepsilon >0$, there exists a positive integer $n_{0}(\varepsilon )$ such
that%
\begin{equation*}
h(x^{n}-x^{m})<\frac{\varepsilon }{2}
\end{equation*}%
for every $m,n\geq n_{0}(\varepsilon )$. By using the definition of $h$, we
obtain that%
\begin{equation*}
\left( \dsum\limits_{k=1}^{\infty }\left( \frac{1}{\lambda _{k}}%
\dsum\limits_{j\in I_{k}}|x_{j}^{n}-x_{j}^{m}|\right) ^{p_{k}}\right)
<\varepsilon ^{M}
\end{equation*}%
for every $m,n\geq n_{0}(\varepsilon ).$ Also we get, for fixed $j\in
\mathbb{N}$, $|x_{j}^{n}-x_{j}^{m}|<\varepsilon $ for every $m,n\geq
n_{0}(\varepsilon )$. Hence it is clear that the sequences $(x_{j}^{n})$ is
a Cauchy sequence in $\mathbb{R}$. Since the real numbers set is complete,
so we have $x_{j}^{m}$ $\rightarrow x_{j}$ for every $n\geq
n_{0}(\varepsilon )$ and as $m\rightarrow \infty $. Now we get%
\begin{equation*}
\left( \dsum\limits_{k=1}^{r}\left( \frac{1}{\lambda _{k}}\dsum\limits_{j\in
I_{k}}|x_{j}^{n}-x_{j}|\right) ^{p_{k}}\right) <\varepsilon ^{M}.
\end{equation*}%
If we pass to the limit over the $r$ to infinity and $n\geq
n_{0}(\varepsilon )$ we obtained $h(x^{n}-x)<\varepsilon $. So, the sequence
$(x^{n})$ is a Cauchy sequence in the space $V(\lambda ;p).$

It remains to show that the space $V(\lambda ;p)$ is complete. Since we have
$x=x^{n}-x^{n}+x$, we get%
\begin{equation*}
\dsum\limits_{k=1}^{\infty }\left( \frac{1}{\lambda _{k}}\dsum\limits_{j\in
I_{k}}|x_{j}|\right) ^{p_{k}}\leq \dsum\limits_{k=1}^{\infty }\left( \frac{1%
}{\lambda _{k}}\dsum\limits_{j\in I_{k}}|x_{j}^{n}-x_{j}|\right)
^{p_{k}}+\dsum\limits_{k=1}^{\infty }\left( \frac{1}{\lambda _{k}}%
\dsum\limits_{j\in I_{k}}|x_{j}^{n}|\right) ^{p_{k}}.
\end{equation*}%
Consequently, we obtain $x\in V(\lambda ;p)$. This completes the proof.

$b)$ By taking $p_{k}=p$ in $(a)$, it can be easily shown the proof of $(b)$.
\end{proof}
\end{theorem}

\section{Some Geometric Properties Of The Spaces $V_{\protect\rho }(\protect%
\lambda ;p)$ And $V_{p}(\protect\lambda ).$}

In this section we construct the modular structure of the space $V(\lambda
;p)$ and since the Luxemburg norm is equivalent to usual norm of the space $%
V_{p}(\lambda )$, we show that the space $V_{p}(\lambda )$ has the
Banach-Saks type $p$.

Firstly, we will introduce a generalized modular sequence space $V_{\rho
}(\lambda ;p)$\ by%
\begin{equation*}
V_{\rho }(\lambda ;p):=\left\{ x\in \ell ^{0}:\rho (\lambda x)<\infty ,\text{
for some }\lambda >0\right\} ,
\end{equation*}%
where%
\begin{equation*}
\rho (x)=\left( \dsum\limits_{k=1}^{\infty }\left( \frac{1}{\lambda _{k}}%
\dsum\limits_{j\in I_{k}}|x_{j}|\right) ^{p_{k}}\right) .
\end{equation*}%
It can be seen that $\rho :V_{\rho }(\lambda ;p)\rightarrow \lbrack 0,\infty
]$ is a modular on $V_{\rho }(\lambda ;p).$

Note that the Luxemburg norm on the sequence space $V_{\rho }(\lambda ;p)$\
is defined as follows:\
\begin{equation*}
||x||_{L}=\inf \left\{ \lambda >0:\rho (\frac{x}{\lambda })\leq 1\right\} ,%
\text{ \ \ \ \ \ \ for all }x\in V_{\rho }(\lambda ;p)
\end{equation*}%
or equally%
\begin{equation*}
||x||_{L}=\inf \left\{ \lambda >0:\rho (\frac{x}{\lambda })=\left(
\dsum\limits_{k=1}^{\infty }\left( \frac{1}{\lambda _{k}}\dsum\limits_{j\in
I_{k}}|x_{j}|\right) ^{p_{k}}\right) \leq 1\right\} .
\end{equation*}%
In the same way we can introduce the Amemiya norm (Orlicz norm) on the
sequence space $\ V_{\rho }(\lambda ;p)$\ as follows:%
\begin{equation*}
||x||_{A}=\inf_{\lambda >0}\frac{1}{\lambda }\left( 1+\rho (\lambda
x)\right) \text{ \ \ \ \ \ for all }x\in V_{\rho }(\lambda ;p).
\end{equation*}%
We now give some basic properties of the modular $\rho $ on the space $%
V_{\rho }(\lambda ;p)$. Also we will investigate some relationships between
the modular $\rho $ and the Luxemburg norm on $V_{\rho }(\lambda ;p).$

\begin{proposition}
\label{prop3.1} The functional $\rho $\ is a convex modular on $V_{\rho
}(\lambda ;p).$
\end{proposition}

\begin{proposition}
\label{prop3.2} For any $x\in V_{\rho }(\lambda ;p)$

i) \textit{if} $\ ||x||_{L}\leq 1$, then $\rho (x)\leq ||x||_{L};$

ii) $||x||_{L}=1$ if and only if $\rho (x)=1.$
\end{proposition}

\begin{proposition}
\label{prop3.3} For any $x\in V_{\rho }(\lambda ;p)$, we have

i) \ If $\ 0<a<1$ and $||x||_{L}>a$, then $\rho (x)>a^{H};$

ii) \textit{if} $\ a\geq 1$ and $||x||_{L}<a,$ then $\rho (x)<a^{H}.$
\end{proposition}

The proofs of the three propositions given above are proved with standard
techniques in a similar way as in \cite{San-Suan-2003} and \cite%
{Cui-Hud-1998}.

\begin{proposition}
\label{prop3.4} Let $(x_{n})$ be a sequence in $V_{\rho }(\lambda ;p).\ $%
Then:

i) if $\underset{n\rightarrow \infty }{\lim }||x_{n}||_{L}=1,$ then $%
\underset{n\rightarrow \infty }{\lim }\rho (x_{n})=1;$

ii) if $\underset{n\rightarrow \infty }{\lim }\rho (x_{n})=0,$ then $%
\underset{n\rightarrow \infty }{\lim }||x_{_{n}}||_{L}=0.$
\end{proposition}

\begin{proof}
(i) Suppose that $\underset{n\rightarrow \infty }{\lim }||x_{n}||_{L}=1.\ $%
Let $\varepsilon \in (0,1).$ Then there exists $\ n_{0}$ \ such that $%
1-\varepsilon <||x_{n}||_{L}<1+\varepsilon $ \ for all \ $n\geq n_{0}.$
Since $(1-\varepsilon )^{H}<||x_{n}||_{L}<(1+\varepsilon )^{H}$ for all \ $%
n\geq n_{0}$ by the Proposition \ref{prop3.3} (i) and (ii), we have\ $\rho
(x_{n})\geq (1-\varepsilon )^{H}$ and $\rho (x_{n})\leq (1-\varepsilon )^{H}$%
. Therefore $\ \underset{n\rightarrow \infty }{\lim }\rho (x_{n})=1.$

(ii) Suppose that $||x_{n}||_{L}\nrightarrow 0$. Then there is an $%
\varepsilon \in (0,1)$ and a subsequence $(x_{n_{_{k}}})$ \ of \ $(x_{n})$
such that $||x_{n_{_{k}}}||_{L}>\varepsilon $ for all $k\in \mathbb{N}$. By
the Proposition \ref{prop3.3} (i), we obtain that $\rho
(x_{n_{k}})>\varepsilon ^{H}$ for all $k\in \mathbb{N}$. This implies that $%
\rho (x_{_{n_{k}}})\nrightarrow 0$ as $n\rightarrow \infty $. Hence \ $\rho
(x_{n})\nrightarrow 0.$
\end{proof}

\begin{theorem}
\label{theo3.5}The space $V_{\rho }(\lambda ;p)$ is a Banach space with
respect to Luxemburg norm defined by%
\begin{equation*}
||x||_{L}=\inf \left\{ \lambda >0:\rho (\frac{x}{\lambda })\leq 1\right\} .
\end{equation*}
\end{theorem}

\begin{proof}
We show that every Cauchy sequence in $V_{\rho }(\lambda ;p)$ is convergent
according to the Luxemburg norm.

Let \ $(x^{n}(j))$\ \ be\ any Cauchy sequence in $V_{\rho }(\lambda ;p)$ and
$\varepsilon \in (0,1).$ Thus, there exists $\ n_{0}$ \ such that $\
||x_{n}-x_{m}||_{L}<\varepsilon ^{M}$ for all $\ m,n\geq n_{0}.$ By the
Proposition 3.2 (i), we obtain%
\begin{equation}
\rho (x^{n}-x^{m})<||x^{n}-x^{m}||_{L}<\varepsilon ^{M},  \label{3.1}
\end{equation}%
for all $\ n,m\geq n_{0},$ that is;%
\begin{equation*}
\dsum\limits_{k=1}^{\infty }\left( \frac{1}{\lambda _{k}}\dsum\limits_{j\in
I_{k}}|x^{n}(j)-x^{m}(j)|\right) ^{p_{k}}<\varepsilon
\end{equation*}%
for\ all\ $m,n\geq n_{0}.$ For fixed\ $j\in \mathbb{N},$ the last inequality
gives that%
\begin{equation*}
|x^{n}(j)-x^{m}(j)|<\varepsilon
\end{equation*}%
for all $m,n\geq n_{0}.$ Hence we obtain that the sequence $(x^{n}(j))$ is\
a Cauchy sequence in $\mathbb{R}$. Since $\mathbb{R}$ is complete, $%
x^{m}(j)\rightarrow x(j)$ as $m\rightarrow \infty .$ Therefore, we have%
\begin{equation*}
\dsum\limits_{k=1}^{\infty }\left( \frac{1}{\lambda _{k}}\dsum\limits_{j\in
I_{k}}|x^{n}(j)-x(j)|\right) ^{p_{k}}<\varepsilon
\end{equation*}%
for all $\ n\geq n_{0}.$

It remains to show that the sequence $(x(j))$ is an element of $V_{\rho
}(\lambda ;p).$ From the inequality (\ref{3.1}), we can write%
\begin{equation*}
\dsum\limits_{k=1}^{\infty }\left( \frac{1}{\lambda _{k}}\dsum\limits_{j\in
I_{k}}|x^{n}(j)-x^{m}(j)|\right) ^{p_{k}}<\varepsilon
\end{equation*}%
for all \ $m,n\geq n_{0}.$ For every $j\in \mathbb{N}$, we have $%
x^{m}(j)\rightarrow x(j)$, so we obtain that%
\begin{equation*}
\rho (x^{n}-x^{m})\rightarrow \rho (x^{n}-x)
\end{equation*}%
as $m\rightarrow \infty $. Since for all\ \ $n\geq n_{0},$%
\begin{equation*}
\dsum\limits_{k=1}^{\infty }\left( \frac{1}{\lambda _{k}}\dsum\limits_{j\in
I_{k}}|x^{n}(j)-x^{m}(j)|\right) ^{p_{k}}\rightarrow
\dsum\limits_{k=1}^{\infty }\left( \frac{1}{\lambda _{k}}\dsum\limits_{j\in
I_{k}}|x^{n}(j)-x(j)|\right) ^{p_{k}}
\end{equation*}%
as$\ m\rightarrow \infty $, then by (\ref{3.1}) we have $\rho
(x^{n}-x)<\left\Vert x^{n}-x\right\Vert _{L}<\varepsilon $ \ for all $n\geq
n_{0}$. This means that $x_{n}\rightarrow x$\ \ as $n\rightarrow \infty $.
So, we have $(x_{n_{0}}-x)\in V_{\rho }(\lambda ;p)$. Since $V_{\rho
}(\lambda ;p)$\ is a linear space, we have $x=x_{n_{0}}-(x_{n_{0}}-x)\in
V_{\rho }(\lambda ;p).$ Therefore the sequence space $V_{\rho }(\lambda ;p)$
is a Banach space with respect to Luxemburg norm. This completes the proof.
\end{proof}

Next, we will show that the space $V_{\rho }(\lambda ;p)$ has Kadec-Klee
property. To do this, we need the following Proposition.

\begin{proposition}
\label{prop3.6}Let $\ x\in V_{\rho }(\lambda ;p)$ \ and $\ (x_{n})\subseteq
V_{\rho }(\lambda ;p).$ If $\ \rho (x_{n})\rightarrow \rho (x)$ as $%
n\rightarrow \infty $ and $x_{n}(j)\rightarrow x(j)$ as $n\rightarrow \infty
$ \ for all $j\in \mathbb{N}$, then $x_{n}\rightarrow x$ as $n\rightarrow
\infty .$
\end{proposition}

\begin{proof}
Let $\varepsilon >0.$ Since $\ \rho (x)=\dsum\limits_{k=1}^{\infty }\left(
\frac{1}{\lambda _{k}}\dsum\limits_{j\in I_{k}}|x(i)|\right) ^{p_{k}}<\infty
$, there exists $j\in \mathbb{N}$ such that%
\begin{equation}
\sum\limits_{k=n_{0}+1}^{\infty }\left( \frac{1}{\lambda _{k}}%
\dsum\limits_{j\in I_{k}}|x(j)|\right) ^{p_{k}}<\frac{\varepsilon }{6K}
\label{3.2}
\end{equation}%
where $K=\max \{1,2^{H-1}\}.$

Since $\rho (x_{n})-\sum\limits_{k=1}^{n_{0}}\left( \frac{1}{\lambda _{k}}%
\dsum\limits_{j\in I_{k}}|x_{n}(j)|\right) ^{p_{k}}\rightarrow \rho
(x)-\sum\limits_{k=1}^{n_{0}}\left( \frac{1}{\lambda _{k}}\dsum\limits_{j\in
I_{k}}|x(j)|\right) ^{p_{k}}$ as $n\rightarrow \infty $\ and$\
x_{n}(j)\rightarrow x(j)$ \ as $n\rightarrow \infty $ \ for all $j\in
\mathbb{N}$, there exists $n_{0}\in \mathbb{N}$ such that%
\begin{equation}
\left\vert \sum\limits_{k=n_{0}+1}^{\infty }\left( \frac{1}{\lambda _{k}}%
\dsum\limits_{j\in I_{k}}|x_{n}(j)|\right)
^{p_{k}}-\sum\limits_{k=n_{0}+1}^{\infty }\left( \frac{1}{\lambda _{k}}%
\dsum\limits_{j\in I_{k}}|x(j)|\right) ^{p_{k}}\right\vert <\frac{%
\varepsilon }{3K}  \label{3.3}
\end{equation}%
for all $n\geq n_{0}.$ Also, since $x_{n}(j)\rightarrow x(j)$ for all $j\in
\mathbb{N}$, we have $\rho (x_{n})\rightarrow \rho (x)$ as $n\rightarrow
\infty $. Hence for all $n\geq n_{0},$\ we have $|x_{n}(j)-x(j)|<\varepsilon
.$ As a result, for all $n\geq n_{0},$ we have%
\begin{equation}
\sum\limits_{k=1}^{n_{0}}\left( \frac{1}{\lambda _{k}}\dsum\limits_{j\in
I_{k}}|x_{n}(j)-x(j)|\right) ^{p_{k}}<\frac{\varepsilon }{3}.  \label{3.4}
\end{equation}%
Then from (\ref{3.2}), (\ref{3.3}) and (\ref{3.4}) it follows that for $%
n\geq n_{0},$%
\begin{eqnarray*}
\rho (x_{n}-x) &=&\sum\limits_{k=1}^{\infty }\left( \frac{1}{\lambda _{k}}%
\dsum\limits_{j\in I_{k}}|x_{n}(j)-x(j)|\right) ^{p_{k}} \\
&=&\sum\limits_{k=1}^{n_{0}}\left( \frac{1}{\lambda _{k}}\dsum\limits_{j\in
I_{k}}|x_{n}(j)-x(j)|\right) ^{p_{k}}+\sum\limits_{k=n_{0}+1}^{\infty
}\left( \frac{1}{\lambda _{k}}\dsum\limits_{j\in
I_{k}}|x_{n}(j)-x(j)|\right) ^{p_{k}} \\
&<&\frac{\varepsilon }{3}+K\left[ \sum\limits_{k=n_{0}+1}^{\infty }\left(
\frac{1}{\lambda _{k}}\dsum\limits_{j\in I_{k}}|x_{n}(j)|\right)
^{p_{k}}+\sum\limits_{k=n_{0}+1}^{\infty }\left( \frac{1}{\lambda _{k}}%
\dsum\limits_{j\in I_{k}}|x(j)|\right) ^{p_{k}}\right] \\
&=&\frac{\varepsilon }{3}+K\left[ \rho
(x_{n})-\sum\limits_{k=1}^{n_{0}}\left( \frac{1}{\lambda _{k}}%
\dsum\limits_{j\in I_{k}}|x_{n}(j)|\right)
^{p_{k}}+\sum\limits_{k=n_{0}+1}^{\infty }\left( \frac{1}{\lambda _{k}}%
\dsum\limits_{j\in I_{k}}|x(j)|\right) ^{p_{k}}\right] \\
&<&\frac{\varepsilon }{3}+K\left[ \rho (x)-\sum\limits_{k=1}^{n_{0}}\left(
\frac{1}{\lambda _{k}}\dsum\limits_{j\in I_{k}}|x_{n}(j)|\right) ^{p_{k}}+%
\frac{\varepsilon }{3K}+\sum\limits_{k=n_{0}+1}^{\infty }\left( \frac{1}{%
\lambda _{k}}\dsum\limits_{j\in I_{k}}|x(j)|\right) ^{p_{k}}\right] \\
&=&\frac{\varepsilon }{3}+K\left[ \sum\limits_{k=n_{0}+1}^{\infty }\left(
\frac{1}{\lambda _{k}}\dsum\limits_{j\in I_{k}}|x(j)|\right) ^{p_{k}}+\frac{%
\varepsilon }{3K}+\sum\limits_{k=n_{0}+1}^{\infty }\left( \frac{1}{\lambda
_{k}}\dsum\limits_{j\in I_{k}}|x(j)|\right) ^{p_{k}}\right] \\
&<&\frac{\varepsilon }{3}+\frac{\varepsilon }{3}+\frac{\varepsilon }{3}%
=\varepsilon .
\end{eqnarray*}%
This shows that $\rho (x_{n}-x)\rightarrow 0$ \ as $n\rightarrow \infty .$\
Hence by Proposition \ref{prop3.4} (ii), we have $||x_{n}-x||_{L}\rightarrow
0$ \ as $\ n\rightarrow \infty .$
\end{proof}

Now, we give one of the main result of this paper involving geometric
properties of the space $V_{\rho }(\lambda ;p).$

\begin{theorem}
\label{theo3.7}The space $V_{\rho }(\lambda ;p)$ has the Kadec-Klee property.
\end{theorem}

\begin{proof}
Let \ $x\in S(V_{\rho }(\lambda ;p))$ and \ $(x_{n})\subseteq B(V_{\rho
}(\lambda ;p))$ \ such that $||x_{n}||_{L}\rightarrow 1$ and $x_{n}\overset{w%
}{\rightarrow }x$ \ as $\ n\rightarrow \infty .$ From Proposition \ref%
{prop3.2} (ii), we have $\rho (x)=1,$ so it follows from Proposition \ref%
{prop3.4}\ (i) that $\rho (x_{n})\rightarrow \rho (x)$ as $\ n\rightarrow
\infty .$ Since $x_{n}\overset{w}{\rightarrow }x$\ and the $\ i^{th}$%
-coordinate mapping $\pi _{j}:V_{\rho }(\lambda ;p)\rightarrow \mathbb{R}$ \
defined by $\ \pi _{j}(x)=x(j)$ \ is continuous linear function on $V_{\rho
}(\lambda ;p),$ it follows that \ $x_{n}(j)\rightarrow x(j)$ \ as $%
n\rightarrow \infty $ for all $j\in \mathbb{N}$. Thus, by Proposition \ref%
{prop3.6} that $x_{n}\rightarrow x$ as $\ n\rightarrow \infty .$
\end{proof}

We prove the following theorem regarding the Banach-Saks of type $p$
property.

\begin{theorem}
\label{theo3.8}The space $V_{p}(\lambda )$ has the Banach-Saks of type $p$.
\end{theorem}

\begin{proof}
From the Theorem \ref{theo2.1} b), it is known that the space $V_{p}(\lambda
)$ is a Banach space with respect to the norm $||x||_{V_{p}(\lambda )}$.

Let $\left( \varepsilon _{n}\right) $ be a sequence of positive numbers for
which $\sum\limits_{n=1}^{\infty }\varepsilon _{n}\leq \frac{1}{2}$. Let $%
\left( x_{n}\right) $ be a weakly null sequence in $B(V_{p}(\lambda ))$. Set
$b_{0}=x_{0}=0$ and $b_{1}=x_{n_{1}}=x_{1}$. Then there exists $m_{1}\in
\mathbb{N}$ such that%
\begin{equation*}
\left\Vert \sum\limits_{i=m_{1}+1}^{\infty }b_{1}(i)e^{(i)}\right\Vert
_{V_{p}(\lambda )}<\varepsilon _{1.}
\end{equation*}%
Since $(x_{n})$ is a weakly null sequence implies $x_{n}\rightarrow 0$
(coordinatewise), there is an $n_{2}\in \mathbb{N}$ such that%
\begin{equation*}
\left\Vert \sum\limits_{i=0}^{m_{1}}x_{n}(i)e^{(i)}\right\Vert
_{V_{p}(\lambda )}<\varepsilon _{1},
\end{equation*}%
where $n\geq n_{2}$. Set $b_{2}=x_{n_{2}}$. Then there exists an $%
m_{2}>m_{1} $ such that%
\begin{equation*}
\left\Vert \sum\limits_{i=m_{2}+1}^{\infty }b_{2}(i)e^{(i)}\right\Vert
_{V_{p}(\lambda )}<\varepsilon _{2.}
\end{equation*}%
By using the fact that $x_{n}\rightarrow 0$ (coordinatewise), there exists
an $n_{3}>n_{2}$ such that%
\begin{equation*}
\left\Vert \sum\limits_{i=0}^{m_{2}}x_{n}(i)e^{(i)}\right\Vert
_{V_{p}(\lambda )}<\varepsilon _{2,}
\end{equation*}%
where $n\geq n_{3}.$

If we continue this process, we can find two increasing subsequences $%
(m_{i}) $ and $(n_{i})$ such that%
\begin{equation*}
\left\Vert \sum\limits_{i=0}^{m_{j}}x_{n}(i)e^{(i)}\right\Vert
_{V_{p}(\lambda )}<\varepsilon _{j,}
\end{equation*}%
for each $n\geq n_{j+1}$ and%
\begin{equation*}
\left\Vert \sum\limits_{i=m_{j}+1}^{\infty }b_{j}(i)e^{(i)}\right\Vert
_{V_{p}(\lambda )}<\varepsilon _{j,}
\end{equation*}%
where $b_{j}=x_{n_{j}}$. Hence,%
\begin{equation*}
\left\Vert \sum\limits_{j=0}^{n}b_{j}\right\Vert _{V_{p}(\lambda
)}=\left\Vert \sum\limits_{j=0}^{n}\left(
\sum\limits_{i=0}^{m_{j-1}}b_{j}(i)e^{(i)}+\sum%
\limits_{i=m_{j-1}+1}^{m_{j}}b_{j}(i)e^{(i)}+\sum\limits_{i=m_{j}+1}^{\infty
}b_{j}(i)e^{(i)}\right) \right\Vert _{V_{p}(\lambda )}
\end{equation*}%
\begin{equation*}
\leq \left\Vert \sum\limits_{j=0}^{n}\left(
\sum\limits_{i=m_{j-1}+1}^{m_{j}}b_{j}(i)e^{(i)}\right) \right\Vert
_{V_{p}(\lambda )}+\left\Vert \sum\limits_{j=0}^{n}\left(
\sum\limits_{i=0}^{m_{j-1}}b_{j}(i)e^{(i)}\right) \right\Vert
_{V_{p}(\lambda )}+\left\Vert \sum\limits_{j=0}^{n}\left(
\sum\limits_{i=m_{j}+1}^{\infty }b_{j}(i)e^{(i)}\right) \right\Vert
_{V_{p}(\lambda )}.
\end{equation*}%
\begin{equation*}
\leq \left\Vert \sum\limits_{j=0}^{n}\left(
\sum\limits_{i=m_{j-1}+1}^{m_{j}}b_{j}(i)e^{(i)}\right) \right\Vert
_{V_{p}(\lambda )}+2\sum\limits_{j=0}^{n}\varepsilon _{j}.
\end{equation*}%
On the other hand since,

$\left\Vert x_{n}\right\Vert _{V_{p}(\lambda )}=\left(
\sum\limits_{k=1}^{\infty }\left( \frac{1}{\lambda _{k}}\dsum\limits_{j\in
I_{k}}|x_{n_{k}}\left( j\right) |\right) ^{p}\right) ^{\frac{1}{p}}$, it can
be seen that $\left\Vert x_{n}\right\Vert _{V_{p}(\lambda )}<1.$ Therefore $%
\left\Vert x_{n}\right\Vert _{V_{p}(\lambda )}^{p}<1$. We have%
\begin{eqnarray*}
\left\Vert \sum\limits_{j=0}^{n}\left(
\sum\limits_{i=m_{j-1}+1}^{m_{j}}b_{j}(i)e^{(i)}\right) \right\Vert
_{V_{p}(\lambda )}^{p}
&=&\sum\limits_{j=0}^{n}\sum\limits_{i=m_{j-1}+1}^{m_{j}}\left( \frac{1}{%
\lambda _{i}}\dsum\limits_{v\in I_{i}}|b_{j}(v)|\right) ^{p} \\
&\leq &\sum\limits_{j=0}^{n}\sum\limits_{i=0}^{\infty }\left( \frac{1}{%
\lambda _{i}}\dsum\limits_{v\in I_{i}}|b_{j}(v)|\right) ^{p} \\
&\leq &(n+1).
\end{eqnarray*}%
Hence we obtain,

\begin{equation*}
\left\Vert \sum\limits_{j=0}^{n}\left(
\sum\limits_{i=m_{j-1}+1}^{m_{j}}b_{j}(i)e^{(i)}\right) \right\Vert
_{V_{p}(\lambda )}\leq (n+1)^{\frac{1}{p}}.
\end{equation*}%
By using the fact \ $1\leq (n+1)^{\frac{1}{p}}$ \ for all\ $n\in \mathbb{N}$
\ and $\ 1\leq p<\infty $, we have%
\begin{equation*}
\left\Vert \sum\limits_{j=0}^{n}b_{j}\right\Vert _{V_{p}(\lambda )}\leq
(n+1)^{\frac{1}{p}}+1\leq 2(n+1)^{\frac{1}{p}}.
\end{equation*}%
Hence $V_{p}(\lambda )$ has the Banach-Saks type $p$. This completes the
proof of the theorem.
\end{proof}

\end{document}